\documentclass[12pt,oneside]{amsart}

\usepackage{amsmath,ifthen, amsfonts, amssymb,cite}
\usepackage{graphicx}

\newtheorem{thm}{Theorem}[section]
\newtheorem{lem}[thm]{Lemma}
\newtheorem{cor}[thm]{Corollary}

\newtheorem{prop}[thm]{Proposition}

\theoremstyle{definition}
\newtheorem{defn}[thm]{Definition}

\newtheorem{rem}[thm]{Remark}
\newtheorem{exmp}[thm]{Example}

\begin{document}

\title[$\mathbb{Z}^n$-free groups are CAT(0)]
{$\mathbb{Z}^n$-free groups are CAT(0)}

\author[I.~Bumagin]{Inna Bumagin}
      \address{ School of Mathematics and Statisics\\
               Carleton University \\
               Ottawa, ON, Canada  K1S 5B6}
      \email{bumagin@math.carleton.ca}

\author[O.~Kharlampovich]{Olga Kharlampovich}
      \address{Department of Mathematics and Statisics\\
               McGill University \\
               Montreal, Quebec, Canada H3A 2K6 }
      \email{olga@math.mcgill.ca}
\thanks{Research supported by NSERC grant}

%\date{\today}

\begin{abstract}
We show that every group with free $\mathbb{Z}^n$-length function is CAT(0).
  \end{abstract}

\maketitle

\section{Introduction}
\label{sec:intro}
\begin{defn}\label{defn:Lambdafreegp} Let $\Lambda$ be an ordered abelian group. By a $\Lambda$-\emph{free} group we mean a finitely generated group $G$ equipped with a free Lyndon $\Lambda$-length function. We call $G$
a \emph{regular} $\Lambda$-\emph{free} group if the length function is free and regular (see section~\ref{sec:prelim} for the definitions).
\end{defn}

Length functions were introduced by Lyndon~\cite{Lyndon} (see also~\cite{LS77}) to generalize Nielsen methods for free groups. They generalize the notion of a word metric on a group and remain a valuable tool in the combinatorial  group theory. The Chiswell construction, introduced by Chiswell~\cite{ChiswellConstruction} for groups with $\mathbb{R}$-valued length functions and generalized by Morgan and Shalen~\cite{MS} to $\Lambda$-\emph{free} group for arbitrary ordered group $\Lambda$, relates length functions with group actions on $\Lambda$-trees. The construction allows one to use free group actions and free Lyndon length functions as two equivalent languages describing the same objects. We refer to the book \cite{ChiswellBook} by Chiswell for a detailed discussion on the subject.

Thus, $\Lambda$-\emph{free} groups can be thought of as groups acting freely on $\Lambda$-trees, and regular $\Lambda$-\emph{free} groups are precisely those acting with unique orbit of branch points~\cite{KharlampovichMyasnikov10}.

 Bass-Serre theory of groups acting freely on simplicial trees~\cite{Serre} implies in particular that free groups are precisely $\mathbb{Z}$-free groups in our terminology. According to Rips' theorem, a finitely generated group is $\mathbb{R}$-free if and only if it is a free product of free abelian groups and surface groups (with few exceptions)~\cite{GLP}. Free actions on $\mathbb{R}$-trees  cover all Archimedean actions, since every group acting freely on a $\Lambda$-tree for an Archimedean ordered abelian group $\Lambda$ acts freely also on an $\mathbb{R}$-tree.

Group actions on $\Lambda$-trees for an arbitrary ordered abelian group $\Lambda$ were introduced by Morgan and Shalen~\cite{MS} and studied by Alperin and Bass~\cite{AlperinBass}. Bass~\cite{Bass} proved a version of combination theorem for finitely generated groups acting freely on  ($\Lambda \oplus \mathbb{Z}$)-trees with respect to the right lexicographic order on $\Lambda \oplus \mathbb{Z}$; this was generalized by Martino and O'Rourke~\cite{MartinoRourke}. A structure theorem for finitely generated  groups acting freely on an ${\mathbb R}^n$-tree (with the lexicographic order) was proved by Guirardel~\cite{Guirardel}. The theorem states that a finitely generated $\mathbb{R}^n$-free group is the fundamental group of a graph of groups with cyclic edge subgroups and vertex groups that are finitely generated $\mathbb{R}^{n-1}$-free groups. The structure theorem, along with the combination theorem for relatively hyperbolic groups proved by Dahmani~\cite{Dahmani}, implies that finitely generated $\mathbb{R}^n$-free groups are hyperbolic relative to maximal abelian subgroups. Note that every $\mathbb{Z}^n$-free group is in particular $\mathbb{R}^n$-free, so that finitely generated $\mathbb{Z}^n$-free groups are relatively hyperbolic; we use this result below.

Kharlampovich, Myasnikov, Remeslennikov and Serbin proved necessary and sufficient conditions for a finitely generated group to be regular $\mathbb{Z}^n$-free~\cite{KharlampovichMyasnikov10}. In the present paper we use the 'only if' part - that is, a structure theorem~\cite[Theorem 7]{KharlampovichMyasnikov10} (see also theorems~\ref{thm:main2} and~\ref{thm:singleHNN} in the present paper). The structure theorem describes a finitely generated regular $\mathbb{Z}^n$-free group as a group obtained from a finitely generated free group by a sequence of finitely many HNN-extensions of a particular type. The same authors showed in~\cite{KharlampovichMyasnikov11} that every finitely generated $\mathbb{Z}^n$-free group embeds by a length-preserving monomorphism into a finitely generated regular $\mathbb{Z}^m$-free group, for some $m$. (Notice that by \cite{ChiswellMuller}, every  $\Lambda$-free group embeds by a length-preserving monomorphism into a  regular $\Lambda$-free group.) We use these theorems to prove our main result (cf. Theorems~\ref{thm:regularZnfreecoherent} and~\ref{thm:ZnfreeCAT0}):
\begin{thm}\label{thm:ZnfreeCAT0intro} A finitely generated $\mathbb{Z}^n$-free group $G$ acts properly discontinuously and cocompactly on a CAT(0) space. If the length function $l\colon G\rightarrow\mathbb{Z}^n$ is also regular then $G$ is the fundamental group of a geometrically coherent space; in particular, $G$ is coherent.
\end{thm}
We refer the reader to the book~\cite{BridsonHaefliger} for definitions and basic properties of CAT(0) spaces and groups acting on them.
 Recall that a group $G$ is called \emph{coherent} if every finitely generated subgroup of $G$ is finitely presented.  Definition~\ref{defn:coherentspace} of a geometrically coherent space can be found in section~\ref{sec:gluing}. The coherence of regular $\mathbb{Z}^n$-free groups can also be deduced from the structure theorem~\cite[Theorem 7]{KharlampovichMyasnikov10}, using the Bass-Serre theory.

Since $\mathbb{Z}^n$-free groups are relatively hyperbolic, the criterion~\cite[Theorem 1.2.1]{HruskaKleiner} proved by Hruska and Kleiner applies, and we conclude that the CAT(0) space in theorem~\ref{thm:ZnfreeCAT0intro} has isolated flats in the sense of Hruska~\cite{HruskaThesis}:

\begin{thm}\label{thm:CAT0isolflats} A finitely generated $\mathbb{Z}^n$-free group $G$ acts properly discontinuously and cocompactly on a CAT(0) space with isolated flats.
\end{thm}

It is unknown, whether all hyperbolic groups are CAT(0). Consequently, the question is open for groups hyperbolic relative to abelian parabolic subgroups. Even the case of $\mathbb{R}^n$-free groups with $n\geq 2$ is open, still.

Kharlampovich and Myasnikov proved in~\cite{KharlampovichMyasnikov08} that every finitely generated fully residually free (or limit) group acts freely on a $\mathbb{Z}^n$-tree for some $n$. These groups form a proper subclass of the class of $\mathbb{Z}^n$-free groups, as the following example shows.
\begin{exmp} $G=\langle x_1,x_2,x_3\mid x_1^2x_2^2x_3^2 = 1\rangle$ is an example of a $\mathbb{Z}^2$-free group which is not fully residually free. Indeed, since $G$ splits as the amalgamated product of 2-generated free groups:
\[
G=\langle x_1,x_1^2x_2\rangle\ast_{\langle x_1^2x_2=(x_2x_3^2)^{-1}\rangle}\langle x_2x_3^2,x_3\rangle,
\]
by Bass' Theorem~\cite{Bass} it is $\mathbb{Z}^2$-free. To show that $G$ is not (fully) residually free, note that in a free group, $x_1^2x_2^2$ is not a proper square unless $[x_1,x_2]=1$.
\end{exmp}

For limit groups theorem~\ref{thm:CAT0isolflats} was proven by Alibegovic and Bestvina~\cite{AlibegovicBestvina}. We adapt their approach to the more general situation.

Notice that ${\mathbb Z}^n$-free groups possess better algorithmic properties than relatively hyperbolic groups with abelian parabolics do in general. For example, by Nikolaev's result \cite{Nikol}, ${\mathbb Z}^n$-free groups have decidable membership problem for finitely generated subgroups while there are examples of hyperbolic groups where this problem is undecidable.

\medskip

The paper is organized as follows. Section~\ref{sec:prelim} contains definitions as well as basic properties of free actions on $\mathbb{Z}^n$-trees. The results presented in this section can be found in~\cite{KharlampovichMyasnikov10}, where the authors use the language of infinite words. In Section~\ref{sec:structure} we construct a weighted word metric on a given regular $\mathbb{Z}^n$-free group $G$, to be used in the proof of theorem~\ref{thm:ZnfreeCAT0intro} in Section~\ref{sec:gluing}.

%%%%%%%%%%%%%%%%%%%%%%%%%%%%%%%%%%

%%%%%%%%%%%%%%%%%%%%%%%%%%%%%%%%%%%%%%%
%\input{LengthFunctionsActions}
%
\section{Length functions and actions}
\label{sec:prelim}

\subsection{Lyndon length functions and free actions on trees}
\label{subsec:Lyndon}

Let $G$ be a group and $\Lambda$ an ordered
abelian group. Then a function $l: G \rightarrow \Lambda$ is called a {\it
(Lyndon) length function} on $G$ if the following conditions hold:
\begin{enumerate}
\item [(L1)] $\forall\ g \in G:\ l(g) \geqslant 0$ and $l(1)=0$;
\item [(L2)] $\forall\ g \in G:\ l(g) = l(g^{-1})$;
\item [(L3)] $\forall\ g, f, h \in G:\ c(g,f) > c(g,h)
\rightarrow c(g,h) = c(f,h)$,

\noindent where $c(g,f) = \frac{1}{2}(l(g)+l(f)-l(g^{-1}f))$.
\end{enumerate}
Sometimes we refer to length functions with values in $\Lambda$ as to $\Lambda$-length functions.

%It is not difficult to derive the following two properties of length functions from the axioms (L1)-(L3):
%\begin{itemize}
% \item $\forall\ g, f \in G:\ l(g f) \leqslant l(g) + l(f)$;
% \item $\forall\ g, f \in G:\ 0 \leqslant c(g,f) \leqslant min\{l(g),l(f)\}$.
%\end{itemize}

While $c(g,f)$ may not belong to $\Lambda$ in general, we are interested in the case when it does.
\begin{defn} \textbf{(Free Length Function)}
A length function $l:G \rightarrow A$ is called {\em free} if it satisfies the following two axioms.
\begin{enumerate}
\item [(L4)] $\forall\ g, f \in G:\ c(g,f) \in \Lambda.$
\item [(L5)] $\forall\ g \in G:\ g \neq 1 \rightarrow l(g^2) > l(g).$
\end{enumerate}
\end{defn}

Moreover, the following particular case is of crucial importance.
\begin{defn} \textbf{(Regular Length Function)} A length function $l: G \rightarrow \Lambda$ is  called {\it regular} if it satisfies the {\it regularity} axiom:
\begin{enumerate}
\item [(L6)] $\forall\ g, f \in G,\ \exists\ u, g_1, f_1 \in G:$
$$g = u g_1,\ l(g)=l(u)+l(g_1);\  f = u f_1,\ l(f)=l(u)+l(f_1); \ l(u) = c(g,f).$$
\end{enumerate}
\end{defn}
Many examples of groups with regular free length functions are given in \cite{KharlampovichMyasnikov10}.

\subsection{$\Lambda$-trees}

Group actions on $\Lambda$-trees provide a geometric counterpart to $\Lambda$-length functions. To explain we need the following definitions, which we also use later.

Let $X$ be a non-empty set, $\Lambda$ an ordered abelian group. An $\Lambda$-{\em metric on} $X$ is a mapping
$d\colon X \times X\longrightarrow \Lambda$ that satisfies the usual axioms of a metric.
The pair $(X,d)$ is an {\em $\Lambda$-metric space}. If $(X,d)$ and
$(X',d')$ are $\Lambda$-metric spaces, an {\it isometry} from
$(X,d)$ to $(X',d')$ is a mapping $f: X \rightarrow X'$ such that
$d(x,y) = d'(f(x),f(y))$ for all $x,y \in X$.

For elements $a,b \in \Lambda$  the {\it closed segment} $[a,b]$ is
defined by
$$[a,b] = \{c \in \Lambda \mid a \leqslant c \leqslant b \}.$$
More generally, a {\it segment} in an $\Lambda$-metric space $X$ is the image of an
isometry $$\alpha: [a,b] \rightarrow X$$ for some $a,b \in
\Lambda$. We denote the segment by $[\alpha(a), \alpha(b)]$; $\alpha(a)$ and $\alpha(b)$ are its endpoints. The length of the segment is $d(\alpha(a),\alpha(b))=b-a$.  A $\Lambda$-metric space $(X,d)$ is
{\it geodesic} if for all $x,y \in X$, there is a segment of length $d(x,y)$ in $X$ with the endpoints $x,y$.

A {\em $\Lambda$-tree} is a $\Lambda$-metric space $(X,d)$ such that:
\begin{enumerate}
\item[(T1)] The space $(X,d)$ is geodesic,
\item[(T2)] If two  segments in $X$ intersect in a single point, which
is an endpoint of both, then their union is a segment,
\item[(T3)] The intersection of two segments with a common endpoint is also a segment.
\end{enumerate}

\subsection{Actions on trees}
From now on, we assume for simplicity that $\Lambda$ is a discrete abelian group. Let a group $G$ act on a $\Lambda$-tree $T$ by isometries. Recall that an element $g\in G$ is called \emph{elliptic} if $g$ fixes a point in $T$, $g$ is an \emph{inversion} if $g$ does not fix a point in $T$ but $g^2$ does, and $g$ is \emph{hyperbolic} if $g$ is neither an elliptic element, nor an inversion.

 An action of $G$ on $X$ is termed {\em free} if for every $1\neq g \in G$ neither  $g$ nor $g^2$ has a fixed point in $X$; in other words, an action of $G$ on $T$ is free if and only if every nontrivial element of $G$ is hyperbolic.

Let a group $G$ act on a $\Lambda$-tree $T$ with a base point $x$, and let the length function $l_x\colon G\rightarrow \Lambda$ based at $x$ be defined as follows:
$l_x(g) = d(x,gx),\forall\ g \in G$. If $l_x$ is a free regular Lyndon length function on $G$ then the action is free, and, according to~\cite[Lemma 1]{KharlampovichMyasnikov10}, $\forall g,f\in G$ $\exists w\in G$ so that $[x,gx]\cap[x,fx]=[x,wx]$. The element $w$ is called the \emph{common beginning} of $f$ and $g$, denoted by $com(f,g)=w$.

 Along with the notion of a length function $l\colon G\rightarrow \Lambda$ we use the \emph{translation length function }$l_T\colon G\rightarrow \Lambda$, defined as follows:
\[
l_T(g)=\min_{t\in T}d(t,gt).
\]
The minimum is always attained, and $g$ is elliptic iff $l_T(g)=0$. Every hyperbolic element $g\in G$ has an \emph{axis} which we denote by $A_g$. %
%The axis is defined as follows: fix a point $p$ such that $d(p,gp)=l_T(g)$, then
%$$A_g=\bigcup_{n=-\infty}^\infty [g^np,g^{n+1}p].$$
%
The axis of $g$ is the set of all the points in $T$ that $g$ moves by the shortest possible distance:
\[
A_g=\{t\in T\mid d(t,gt)=l_T(g)\}.
\]

Note that while the values $l(g)$ of the length function depend on the choice of a base point $x\in T$, the values of the translation length function $l_T(g)$ do not. For every hyperbolic element $g\in G$, $l(g)\geq l_T(g)$, and $l(g)=l_T(g)$ iff $x\in A_g$.

Unlike $l(g)$, the translation length is a conjugacy class invariant, so that
\[
l_T(a^{-1}ga)=l_T(g), \ \forall a\in G.
\]

%%%%%%%%%%%%%%%%%%%%%%%%%%%%%%%%%%%%%%%%%%%
%%%    Properties of complete Z^n-trees
%%%%%%%%%%%%%%%%%%%%%%%%%%%%%%%%%%%%%%%%%%%

\begin{lem} \label{lem:axesthroughx} Let $G$ act freely on a $\Lambda$-tree $T$, and let us fix a base point $x\in T$. If the associated length function $l\colon G\rightarrow \Lambda$ is regular then every $g\neq 1$ in $G$ is conjugate to an element $f\in G$ so that $x\in A_f$.
\end{lem}
\begin{proof} Assume that $x\notin A_g$. Let $y\in A_g$ be the point on $A_g$ closest to $x$. Note that $[x,y]$=$[x,gx]\cap[x,g^{-1}x]$, so that $d(x,y)=c(g,g^{-1})$. By the regularity of $l$, there is $u\in G$ so that $y=ux$. Then $x=u^{-1}y$, and hence $x\in u^{-1}A_g = A_{u^{-1}gu}$. We set $f=u^{-1}gu$ to obtain the claim.
\end{proof}

In what follows, we will only deal with those representatives of conjugacy classes that are shortest with respect to $l$; Lemma~\ref{lem:axesthroughx} gives rise to the following definition.
\begin{defn}\label{defn:cyclicallyreduced}
We say that a hyperbolic element $g$ of a regular $\Lambda$-free group $G$ is \emph{cyclically reduced} if $[x,gx]\cap[x,g^{-1}x]=\{x\}$.
\end{defn}
If, for instance, $F$ is a finitely generated free group acting freely on its Cayley graph $T$, the associated length function is clearly regular. In this case, our notion of a cyclically reduced element coincides with the standard notion~\cite{LS77}. Without loss of generality, in this case we can assume that the base point of $T$ corresponds to the identity of $F$.

\subsection{Basic properties of $\mathbb{Z}^n$-free groups}
From  now on, we consider free actions on $\mathbb{Z}^n$-trees, so that the length function $l$ is free and takes its values in $\mathbb{Z}^n$ with the right lexicographic order. If $n>1$, so that the length is not Archimedean, following~\cite{KharlampovichMyasnikov10}, we introduce the notion of height, as follows.
\begin{defn}\label{defn:height} \textbf{(Height of element)} Let $l\colon G\rightarrow\mathbb{Z}^n$ be a length function. By the \emph{height} of an element $g\in G$, denoted $ht(g)$, we mean the number of the rightmost non-zero component of $l(g)$. If $l(g)=0$ then we say that $ht(g)=0$. By the height of a set $S\subset G$ we mean the maximum height of its elements: $$ht(S)=\max_{g\in S}\{ht(g)\}. $$
\end{defn}
If $ht(g)>ht(f)$ for some $g,f\in G$ then we say that $l(g)$ is infinitely larger than $l(f)$ (or $l_T(g)$ is infinitely larger than $l_T(f)$), and write $l(g)\gg l(f)$ (or $l_T(g)\gg l_T(f)$). We call $l(g)$ and $l(f)$ (or $l_T(g)$ and $l_T(f)$) \emph{comparable} if $ht(g)=ht(f)$, so that none of the two lengths is infinitely larger than the other.

%For a general ordered abelian group $\Lambda$, we write $l(g)\gg l(f)$ (or $l_T(g)\gg l_T(f)$) whenever $l(g)$ is infinitely larger than $l(f)$ (or $l_T(g)$ is infinitely larger than $l_T(f)$). For %$\Lambda=\mathbb{Z}^n$, $l(g)\gg l(f)$ is equivalent to $ht(g)>ht(f)$.  We call $l(g)$ and $l(f)$ (or $l_T(g)$ and $l_T(f)$) \emph{comparable} if none of them is infinitely larger than the other. For %$\Lambda=\mathbb{Z}^n$, $l(g)$ and $l(f)$ are comparable iff $ht(g)=ht(f)$.
%%%%%%%%%%%%%

Let $F$ be a free group - that is, $F$ is $\mathbb{Z}$-free. If $u\in F$ is cyclically reduced and $sus^{-1}=v$ for some $v,s\in F$, then $l(v)>l(u)$, unless $s=\hat{s}u^k$, where $\hat{s}^{-1}$ is an initial subword of $u$ and $k$ is an integer, hence $v$ is a cyclic permutation of $u$ and $l(v)=l(u)$. In Lemma~\ref{lem:axisofstableletter} below we study generalizations of this.

%\begin{defn} (Cyclically reduced elements, cyclic permutations.)
%\end{defn}
%
\begin{lem}\label{lem:axisofstableletter} Let $G$ be a $\mathbb{Z}^n$-free group, and let $u$ and $v$ be cyclically reduced elements of $G$ such that $l(u)=l(v)$. Suppose that $sus^{-1}=v$ for some nontrivial element $s\in G$. Then one of the following holds:
\begin{enumerate}
 \item[Case 1.] \label{e:shorts}  $ht(s)\leq ht(u)$. Then there is $\hat{s}\in s \langle u\rangle$ so that $[x,\hat{s}^{-1}x]\subset[x,ux]$ and $[x,\hat{s}x]\subset[x,v^{-1}x]$.
 \item[Case 2.] \label{e:longs} $ht(s) > ht(u)$. Then $[x,v^px]\subset [x,sx]$ and $[x,u^{-p}x]\subset [x,s^{-1}x]$, either for all non-negative integers $p$, or for all non-positive integers $p$.
\end{enumerate}
\end{lem}
\begin{proof} The assumption that both $u$ and $v$ are cyclically reduced and conjugate in $G$ implies that the axes $A_u$ of $u$ and $A_v$ of $v$ contain the base point $x$, and
$$l(u)=l_T(u)=l_T(v)=l(v).$$ In particular, $A_v=A_{sus^{-1}}=sA_u$ contains both $x$ and $sx$.

\vspace{.2cm}
\emph{Case 1}, $ht(s) \leq ht(u)$. In this case, for some integer $k$,
\begin{equation}\label{e:shortsinclusion}
x\in[su^kx,su^{m}x]=[su^ks^{-1}(sx),su^{m}s^{-1}(sx)],
\end{equation}
where $m=k+1$ if $k> 0$, $m=k-1$ if $k<0$, and $m$ can have either of the two values if $k=0$.
If $k\geq 0$ and $s^{-1}x\in[u^kx,u^{k+1}x]$, we set $\hat{s}=su^k$. If $k\leq 0$ and $s^{-1}x\in[u^{k-1}x,u^kx]$ then set $\hat{s}=su^{k-1}$. In any case, $[x,\hat{s}^{-1}x]\subset[x,ux]$. Note that in a free group this means that $\hat{s}^{-1}$ is an initial subword of $u$, and hence $v$ is a cyclic permutation of $u$.

Note that $\hat{s}u\hat{s}^{-1}=sus^{-1}=v$, and so $\hat{s}x\in A_v$, by the same argument as above. Since $d(x,\hat{s}^{-1}x)=d(\hat{s}x,x)$, necessarily $[x,\hat{s}x]\subset[x,v^{\varepsilon}x]$ with $\varepsilon=1$ or $-1$. However, $\hat{s}u=v\hat{s}$; in particular, $(\hat{s}u)x=(v\hat{s})x$. We compute
\[
d(x,(\hat{s}u)x)=d(\hat{s}^{-1}x,ux)=l(u)-d(\hat{s}^{-1}x,x)=l(u)-d(\hat{s}x,x)=l(v)-d(\hat{s}x,x).
\]
Since $d(x,(v\hat{s})x)=l(v)-d(\hat{s}x,x)$ if $\varepsilon=-1$, and $d(x,(v\hat{s})x)=l(v)+d(\hat{s}x,x)$ if $\varepsilon=1$, we obtain the claim.

%If $s$ and $u$ commute then $v=u$ and $sA_u=A_u=s^{-1}A_u$, in particular $s^{-1}x\in A_u$. Since %$d(x,sx)=d(x,s^{-1}x)$, $sx\neq s^{-1}x$, and so the tripod with the vertices $x,sx,s^{-1}x$ is actually the segment %$[s^{-1}x,sx]\subset A_u$, we conclude that $[x,sx]\cap[x,s^{-1}x]=\{x\}$. It follows that similar %to~(\ref{e:shortsinclusion}), there is the following inclusion:
% \begin{equation}\label{e:commutativeinclusion}
%x\in[s^{-1}u^{-k}x,s^{-1}u^{-m}x]=[s^{-1}(u^{-k}x),s^{-1}(u^{-m}x)]\ \Leftrightarrow\ sx\in [u^{-k}x,u^{-m}x],
%\end{equation}
%and for the same choice of $\hat{s}$ as above, $[x,\hat{s}x]\subset[x,u^{-1}x]=[x,v^{-1}x]$.

\vspace{.2cm}
\emph{Case 2}, $ht(s)>ht(u)=ht(v)$. Since $A_v$ contains both $x$ and $sx$, it follows that either $[x,vx]\subset[x,sx]$, or $[x,v^{-1}x]\subset[x,sx]$. Note that for any integer $p$, $d(x,v^px)\ll d(x,sx)$. Hence, $[x,v^px]\subset [x,sx]$ either for all non-negative integers $p$, or for all non-positive integers $p$. Since $u=s^{-1}vs$, similar arguments show that
$[x,u^qx]\subset [x,s^{-1}x]$ either for all non-negative integers $q$, or for all non-positive integers $q$.

Let $s$ and $u$ commute, then $v=u$ and also $A_s=A_u$ (see Lemma~\ref{lem:reducedcentralizer} for the latter equality). Assume that $[x,u^{\varepsilon}x]\subset[x,sx]$ where $\varepsilon=1$ or $-1$. Then $[x,u^{\varepsilon}x]\subset[x,s^{-1}x]$ would mean that $s$ is not cyclically reduced and would contradict the equality $A_s=A_u$. Necessarily we have that $[x,u^{-\varepsilon}x]\subset[x,s^{-1}x]$, and the claim follows.

It remains to show that $q=-p$ also when $s$ and $u$ do not commute. Let $[x,vx]\subset[x,sx]$, we write this as $[x,sus^{-1}x]\subset[x,sx]$, which amounts to
$ [s^{-1}x,us^{-1}x]\subset[s^{-1}x,x]$. In particular, $us^{-1}x\in[s^{-1}x,x]$. If we suppose now that $[x,ux]\subset[x,s^{-1}x]$, so that $ux\in[x,s^{-1}x]$, then $[ux,us^{-1}x]\subsetneqq[x,s^{-1}x]$, which implies that $d(x,s^{-1}x)=d(ux,us^{-1}x)<d(x,s^{-1}x)$, a contradiction. Necessarily, we have that $[x,u^{-1}x]\subset[x,s^{-1}x]$, so that $[x,v^px]\subset [x,sx]$ implies $[x,u^{-p}x]\subset[x,s^{-1}x]$ for all non-negative integers $p$. The case when $p$ is non-positive is similar.
%%%%%%%%%%%%%%%%
% old version, before the 2nd revision
%%%%%%%%%%%%%%%%%%%%
%If $[x,ux]\subset[x,sx]$ then $[x,u^px]\subset[x,sx]$ for all $p\geq 0$, while $[x,u^{-1}x]\subset[x,sx]$ implies that $[x,u^px]\subset[x,sx]$ for all $p\leq 0$.
%
%Since $s^{-1}vs=u$, similar arguments show that $[x,vx]\subset[x,s^{\epsilon}x]$ where $\epsilon$ is either 1, or $-1$. W.l.o.g., assume that $[x,u^px]\subset[x,sx]$ for all $p\geq 0$. If also $[x,vx]\subset[x,sx]$ then necessarily $[x,vx] = [x,ux]$, as $l(u)=l(v)$. It follows that $v^{-1}ux=x$, and since the action is free, this implies $u=v$;  that is, $s$ and $u$ commute. If $s$ and $u$ do not commute, and so $u\neq v$, we must have $[x,vx]\subset[x,s^{-1}x]$.
\end{proof}

%%%%%%%%%%%%%%%%%%%%%%%%%%%%%%%%%%%%%%
%%%% Abelian Subgroups
%%%%%%%%%%%%%%%%%%%%%%%%%%%%%%%%%%%

\subsection{Abelian Subgroups.}\label{sec:abeliansbgps}

\begin{lem}\label{lem:reducedcentralizer} Let $G$ be a $\mathbb{Z}^n$-free group, $u,s\in G$ be nontrivial. If $u$ and $s$ commute then $A_u=A_s$. In particular, $u$ is cyclically reduced if and only if $s$ is cyclically reduced.
\end{lem}
For a proof see~\cite[Ch.3, Lemma 3.9(2)]{ChiswellBook}. Note that the converse of Lemma~\ref{lem:reducedcentralizer} also holds for free actions because, if  $A_u=A_s$ then the commutator $[u,s]$ fixes every point of $A_u$.

\begin{cor}\label{cor:reducedcentralizer} let $G$ be a regular $\mathbb{Z}^n$-free group, and let $\tilde{A}$ be an abelian subgroup of $G$. Then $\tilde{A}$ is conjugate to an abelian subgroup $A$ of $G$ such that every element of $A$ is cyclically reduced.
\end{cor}
\begin{proof} This follows from Lemmas~\ref{lem:axesthroughx} and ~\ref{lem:reducedcentralizer}.
\end{proof}
%
%\begin{defn}\label{defn:cyclreducedabeliansbgp} An abelian subgroup of a regular $\mathbb{Z}^n$-free group is called \emph{cyclically reduced} if every element of $A$ is cyclically reduced. Note that %in~\cite{KharlampovichMyasnikov10} such a subgroup is termed a \emph{cyclically reduced centralizer}.
%\end{defn}
%

\begin{defn}\label{defn:coherentdirections} Let $G$ act on a $\Lambda$-tree $T$ with a base point $x$. Let $f$ and $g$ be hyperbolic elements with overlapping axes so that $A_g\cap A_f$ contains more than one point. We say that $f$ and $g$ have \emph{coherent directions}, or simply that $f$ and $g$ are \emph{coherent} if there is a point $p\in A_f\cap A_g$ so that $p\notin [gp,fp]$. If $A_g\cap A_f$ is a single point, we regard $g$ and $f$ as coherent.
\end{defn}

\begin{thm}\label{thm:rankabelian} Let $G$ be a regular $\mathbb{Z}^n$-free group, and let $A$ be a cyclically reduced abelian subgroup of $G$. Then $A$ has a generating set $\{t_1,t_2,\dots,t_k\}$ so that $ht(t_i)<ht(t_{i+1})$ for all $i=1,\dots, k-1$. In particular, every abelian subgroup of $G$ can be generated by at most $n$ elements.
\end{thm}
\begin{proof} Let $k=\max_{a\in A}\{ht(a)\}$, we call $k$ the height of $A$ and denote it by $k=ht(A)$. Note that $k=\max_{s\in S}\{ht(s)\}$, for any generating set $S$ for $A$.
 We argue by induction on $k$.

 Let $k=1$. All the elements of height 1 in $G$ form a $\mathbb{Z}$-free subgroup $H$. Hence, $A$ is an abelian subgroup of the free group $H$, and so $A$ is cyclic.

 Now, assume that every abelian subgroup $A_k$ of $G$ of height at most $k$ can be generated by at most $ht(A_k)$ elements, and let $A$ have height $k+1$. Let $S$ be a generating set for $A$, and let $s,t\in S$ have the maximum height: $ht(s)=ht(t)=k+1$. By replacing $s$ by its inverse if necessary, we can assume that $s$ and $t$ are coherent. We claim that $\langle s,t\rangle=\langle p,\tilde{s},\tilde{t}\ \rangle$ for some $p$,$\tilde{s}$ and $\tilde{t}$ with $ht(\tilde{s}),ht(\tilde{t})<ht(p)=k+1$. Indeed, let $m_s>0$ and $m_t>0$ be the largest non-zero components of $l(s)$ and $l(t)$, respectively. We set
%\begin{equation}\label{e:gcd}
$m=\sigma m_s+\tau m_t=gcd(m_s,m_t),\quad p=s^{\sigma }t^{\tau}$.
%\end{equation}
Then we obtain the claim for $$\tilde{s}=sp^{-\frac{m_s}{m}},\quad \tilde{t}=tp^{-\frac{m_t}{m}}.$$
If $m=m_s$ then $p=s$, and we set $\tilde{s}=1$, so that $\langle s,t\rangle=\langle p, \tilde{t}\ \rangle$. Since the $k+1$st component of $l(p)$ equals $m\geq 1$, $ht(p)=k+1$. In the generating set $S$ we replace $s$ and $t$ with $p,\tilde{s}$ and $\tilde{t}$. Clearly, the new set generates $G$ also. If there are more elements of height $k+1$ in $S$, we can proceed replacing them in a similar way until there is only one element $t_{k+1}$ of the maximal height in the obtained set $S'$. Let $S_k=S'\setminus\{t_{k+1}\}$. Denote $\langle S_k\rangle$ by $A_k$, then $A_k\subset A$ with $ht(A_k)\leq k$. By the induction hypothesis, $A_k$ is generated by $\{t_{j_1},t_{j_2},\dots,t_{j_l}\}$ with $l\leq k$ and $ht(t_{j_1})\leq ht(t_{j_2})\leq\dots\leq ht(t_{j_l})\leq k$. Then
$\{t_{j_1},t_{j_2},\dots,t_{j_l},t_{k+1}\}$ is a generating set for $A$ that satisfies the requirements.
\end{proof}

Below we use the following well-known fact about $\mathbb{Z}^n$-free groups.
\begin{lem} Let $G$ be a $\mathbb{Z}^n$-free group. Let $u$ be a nontrivial element of $G$. Then $u$ is not conjugate to $u^{-1}$ in $G$.
\end{lem}
\begin{proof} Suppose $gug^{-1}=u^{-1}$ for some $g\in G$. Take the inverses of both sides of the equality to get $gu^{-1}g^{-1}=u$. On the other hand, the first equality amounts to $u=g^{-1}u^{-1}g$. Therefore, $g^2$ and $u$ commute, so $A_g=A_{g^2}=A_u$ which means that $g$ and $u$ commute. But then $u=u^{-1}$, or $u^{2}=1$, a contradiction since $G$ is torsion-free.
\end{proof}
%%%%%%%%%%%%%%%%%%%%%
In a free group, if $sus^{-1}=v$, where $u$ and $v$ are  cyclically reduced and $s$ is not cyclically reduced, then some cyclic permutation of $u$ and some cyclic permutation of $v$ are conjugate by a cyclically reduced element. The following lemma should be viewed as a generalization of this fact.
\begin{lem} Let $G$ be a regular $\mathbb{Z}^n$-free group with a base point. Let $u$ and $v$ be cyclically reduced elements of $G$ with $l(u)=l(v)>0$, and suppose that $sus^{-1}=v$ for some $s\in G$ such that $s$ is not cyclically reduced. Then there are $w\in G$ and a cyclically reduced $z\in G$ so that at least one of the following holds:
\begin{enumerate}
  \item $u=u_1w^{-1}$, $v=wv_1$, and $zw^{-1}u_1z^{-1}=v_1w$, where $l(u)=l(u_1)+l(w^{-1})$ and $l(v)=l(v_1)+l(w)$.
  \item $u=wu_1$, $v=v_1w^{-1}$, and $zu_1wz^{-1}=w^{-1}v_1$, where $l(u)=l(u_1)+l(w)$ and $l(v)=l(v_1)+l(w^{-1})$.
%  \item $u=wu_1$, $v=wv_1$, and $zu_1wz^{-1}=v_1w$.
\end{enumerate}
If $w\neq 1$ then exactly one of the two cases holds. In each case, $v_1\neq 1$ and $u_1\neq 1$.
\end{lem}
\begin{proof} If $s$ is not cyclically reduced then $[x,s^{-1}x]\cap[x,sx]=[x,wx]$ for some $w\in G$, as the length function is regular. We have that $z=w^{-1}sw$ for some cyclically reduced $z\in G$ (cf. the proof of Lemma~\ref{lem:axesthroughx}), so that $s=wzw^{-1}$.

If $s$ and $u$ commute then by lemma~\ref{lem:reducedcentralizer}, $s$ has to be cyclically reduced. In what follows, we assume that $s$ and $u$ do not commute. We consider the following two cases:

\textbf{Case 1}, $ht(s)>ht(u)$. Assume that $l(w)\geq l(u)$. By Lemma~\ref{lem:axisofstableletter} Case 2, $[x,v^px]\subset[x,sx]$ and $[x,u^{-p}x]\subset[x,s^{-1}x]$ either for all $p\geq 0$ or for all $p<0$. However, $[x,sx]\cap[x,s^{-1}x]=[x,wx]$, and since $l(w)\geq l(u)=l(v)$, we conclude that either $v.x=u^{-1}.x$, or $v^{-1}.x=u.x$; either equality implies $v=u^{-1}$. So, $u$ is conjugate to its inverse, a contradiction.

Now, let $l(w)<l(u)$. We have $[x,wx]\subseteq [x,v^{\mu}x]\cap [x,u^{-\mu}x]$ with $\mu\in\{-1,1\}$. If $\mu=1$ then $v=wv_1$ and $u^{-1}=wu_1^{-1}$, or $u=u_1w^{-1}$, with $l(u)=l(u_1)+l(w^{-1})$ and $l(v)=l(v_1)+l(w)$. Then the equality $sus^{-1}=v$ can be written as
      $
      (wzw^{-1})(u_1w^{-1})(wz^{-1}w^{-1})=wv_1,
      $
      and we have that $zw^{-1}u_1z^{-1}=v_1w$. If $\mu=-1$ then $v^{-1}=wv_1^{-1}$, or $v=v_1w^{-1}$, and $u=wu_1$, with $l(u)=l(u_1)+l(w)$ and $l(v)=l(v_1)+l(w^{-1})$. We have then
      $
      (wzw^{-1})(wu_1)(wz^{-1}w^{-1})=v_1w^{-1},
      $
      or $zu_1wz^{-1}=w^{-1}v_1$. Note that if $w\neq 1$, the cases $\mu=1$ and $\mu=-1$ are mutually exclusive because $u$ and $v$ are reduced, and so $[x,vx]\cap[x,v^{-1}x]=x$ cannot contain $[x,wx]$.

 \textbf{Case 2}, $ht(s)\leq ht(u)$. By Lemma~\ref{lem:axisofstableletter} Case 1, $[x,\hat{s}^{-1}x]\subset[x,ux]$ and $[x,\hat sx]\subset[x,v^{-1}x]$, where $\hat{s}=su^k$ for some $k$. Clearly, $\hat{s}u\hat{s}^{-1}=v$, so if $\hat{s}$ is cyclically reduced then we set $z=\hat{s}$ and $w=1$. Therefore, w.l.o.g. we assume that $[x,s^{-1}x]\subset[x,ux]$ and $[x,sx]\subset[x,v^{-1}x]$, and $s$ is not cyclically reduced. Necessarily, $l(w)<l(s)\leq l(u)$, and so $[x,wx]=[x,ux]\cap[x,v^{-1}x]$. Therefore, $u=wu_1$, and $v^{-1}=wv^{-1}_1$, or $v=v_1w^{-1}$, with $l(u)=l(u_1)+l(w)$ and $l(v)=l(v_1)+l(w^{-1})$, and $zu_1wz^{-1}=w^{-1}v_1$.

Note that $u\neq 1$ and $v\neq 1$ by assumption, so that if $w=1$ then necessarily $v_1\neq 1$ and $u_1\neq 1$. If $w\neq 1$ then the assertion follows because $l(u)=l(v)$ and $v\neq u^{-1}$.
\end{proof}

%%%%%%%%%%%%%%%%%%%%%%%%%%%%%%%%%%%%%%%%%%%%%%%%%%%%%%%%%%%%%%
%\input{StructureProperties}
%

\section{Structure of regular $\mathbb{Z}^n$-free groups} \label{sec:structure}

\begin{lem} \label{lem:sumoflettersthroughbase}
Let $G$ be a finitely generated group acting freely on a $\mathbb{Z}^n$-tree with a base point $x$. Let $g_1,\dots,g_m$ be nontrivial cyclically reduced (not necessarily distinct) elements of $G$, and let $u=g_1\dots g_{m}$. Assume that any two consecutive factors $g_{i}$ and $g_{i+1}$ in the word $u$ are coherent.
Then the axis of $u$ contains $x$, and
$$l(u)=\sum_{i=1}^m l(g_i).$$
\end{lem}

\begin{proof} We prove the assertion by induction on the length $m$ of $u$. Note that every $g\neq 1$ in $G$ is hyperbolic. For $m=1$ the claim is trivial, and for  $m=2$ it is proved in~\cite{ChiswellBook}. The proof shows that under our assumptions, if $f=g_i$ and $g=g_{i+1}$, then $l_T(fg)=l_T(f)+l_T(g)$ and the axis of $fg$ contains the intersection $A_f\cap A_g\subseteq A_{fg}$. Moreover, both intersections $A_f\cap A_{fg}$ and $A_g\cap A_{fg}$ contain more than one point, and $f$ and $fg$, as well as $g$ and $fg$ are coherent. In particular, if we denote by $$u_1=f_1g_3\dots,g_{m},$$ where $f_1=g_1g_2$, both assumptions will hold for $u_1$ as well. The claim now follows by induction.
\end{proof}

\begin{defn}\label{defn:cyclreducedabeliansbgp} An abelian subgroup of a regular $\mathbb{Z}^n$-free group is called \emph{cyclically reduced} if every element of $A$ is cyclically reduced. Note that in~\cite{KharlampovichMyasnikov10} such a subgroup is termed a \emph{cyclically reduced centralizer}.
\end{defn}

\begin{thm} \cite[Theorem 7]{KharlampovichMyasnikov10}
\label{thm:main2}
Let $G$ be a finitely generated group with a regular free Lyndon
length function in $\mathbb{Z}^n$. Then $G$ can be represented as
a union of a finite series of groups
$$G_1 < G_2 < \cdots < G_r = G,$$
where $G_1$ is a free group of finite rank, and
$$G_{i+1} = \langle G_i, s_{i,1},\ \ldots,\ s_{i,k_i} \mid s_{i,j}^{-1}\ C_{i,j}\ s_{i,j} = \phi_{i,j}(C_{i,j}) \rangle,$$
where for each $j \in [1,k_i],\ C_{i,j}$ and $\phi_{i,j}(C_{i,j})$ are cyclically reduced abelian subgroups of $G_i$,
$\phi_{i,j}$ is an isomorphism, and the following conditions are satisfied:
\begin{enumerate}
\item[(1)] $C_{i,j} = \langle c^{(i,j)}_1, \ldots, c^{(i,j)}_{m_{i,j}} \rangle,\
\phi_{i,j}(C_{i,j}) = \langle d^{(i,j)}_1, \ldots, d^{(i,j)}_{m_{i,j}} \rangle$,
where $\phi_{i,j}(c^{(i,j)}_k) = d^{(i,j)}_k,\ k \in [1,m_{i,j}]$ and
$$ht(c^{(i,j)}_k) = ht(d^{(i,j)}_k) < ht(d^{(i,j)}_{k+1}) = ht(c^{(i,j)}_{k+1}),\
k \in [1,m_{i,j}-1],$$
% from the paper
$ht(s_{i,j})>ht(c^{(i,j)}_k)$,
\item[(2)]  $\ell(\phi_{i,j}(w)) =\ell (w)$ for any $w \in C_{i,j}$,
\item[(3)] $w$ is not conjugate to $\phi_{i,j}(w)^{-1}$ in $G_i$ for any $w \in C_{i,j}$,
\item[(4)] if $A, B \in \{C_{i,1}, \phi_{i,1}(C_{i,1}), \ldots, C_{i,k_i},
\phi_{i,k_i}(C_{i,k_i})\}$ then either $A = B$, or $A$ and $B$ are not conjugate
in $G_i$,
\item[(5)] $C_{i,j}$ can appear in the list
$$\{ C_{i,k}, \phi_{i,k}(C_{i,k}) \mid k \neq j \}$$
not more than twice.
\item[(6)] $ht(s_{i,j})=ht(s_{i,k})$, for all $1\leq k,j\leq k_i$.
\end{enumerate}
\end{thm}

%\textbf{Remark. (3) follows from (4) because of malnormality of maximal abelian subgroups.}
\begin{defn}\label{defn:compatibledir}
Let $C_{i,k}$ and $\phi_{i,k}(C_{i,k})$ be as in the statement of Theorem~\ref{thm:main2}, and let $S$ be a subset of elements of $C_{i,k}$. We say that the elements of $S$ have \emph{directions compatible with} $\phi_{i,k}$ if $\forall g,f\in S$, $g$ and $f$ are coherent if and only if $\phi_{i,k}(g)$ and $\phi_{i,k}(f)$ are coherent.
\end{defn}
\begin{cor}\label{cor:coherenceinabeliansbgp} Let $C_{i,j}$ be as in the statement of Theorem~\ref{thm:main2}. Generators for $C_{i,j}$ can be chosen so that they are all coherent in the meaning of Definition~\ref{defn:coherentdirections} and the directions of the generators are compatible with $\phi_{i,j}$.
\end{cor}
\begin{proof} In $C_{i,j} = \langle c^{(i,j)}_1, \ldots, c^{(i,j)}_{m_{i,j}} \rangle$ choose all $c^{(i,j)}_k$ be coherent with $c^{(i,j)}_1$ by replacing $c^{(i,j)}_k$ with its inverse if necessary. In $\phi_{i,j}(C_{i,j})= \langle d^{(i,j)}_1, \ldots, d^{(i,j)}_{m_{i,j}} \rangle$ choose all the generators be coherent with $d^{(i,j)}_1=\phi_{i,j}(c^{(i,j)}_1)$. Then $d^{(i,j)}_k=\phi_{i,j}(c^{(i,j)}_k)$, by the assertions (1) and (2) of Theorem~\ref{thm:main2}. Indeed, if $c$ and $a$ were two coherent generators of $C_{i,j}$, while $\phi_{i,j}(c),\phi_{i,j}(a)\in\phi_{i,j}(C_{i,j})$ were not coherent, then $l(ac)>l(\phi_{i,j}(ac))$ would contradict~(2).
\end{proof}

Observe, that if $l:G \to \Lambda$ is a Lyndon length function with values
in some ordered abelian group $\Lambda$ and $\mu:\Lambda \to \Lambda^\prime$
is a homomorphism of ordered abelian groups then the composition
$l^\prime = \mu \circ l$ gives a Lyndon length function
$l^\prime : G \to \Lambda ^\prime$. In particular, since ${\mathbb Z}^{n-1}$ is a
convex subgroup of ${\mathbb Z}^n$ then the canonical projection $\pi_n:{\mathbb Z}^n \to {\mathbb Z}$
such that  $\pi_n(x_1,x_2,\ldots,x_n) = x_n$ is an ordered homomorphism, so
the composition $\pi_n \circ l$ gives a Lyndon length function
$l^\prime : G \to {\mathbb Z}$ such that $l^\prime(g) = \pi_n(l(g))$. Notice also that
if $u = g \circ h$ then $l^\prime(u) = l^\prime(g) + l^\prime(h)$ for any
$g, h, u \in G$.

If $Y$ is a finite set of elements in  the group $G$ acting freely and regularly on ${\mathbb Z}^n$-tree, we denote by $Y_+$ the set of elements with non-zero length $l^\prime$ and by $Y_0$ the rest of elements.

\begin{defn} \textbf{(Reduced generating set)}\label{defn:reducedset}\cite[Definition 2]{KharlampovichMyasnikov10}.
A finite set $Y$ of elements of the group $G$ acting freely and regularly on ${\mathbb Z}^n$-tree is called \emph{reduced} if it satisfies the following conditions:
\begin{itemize}
\item[(a)] all elements of the subset $Y_+$  are cyclically reduced;
\item[(b)] if $f,g \in Y_+^{\pm 1},\ f \neq g$ then $l^\prime(com(f,hg)) = 0$
for any $h \in \langle Y_0 \rangle$;
\item[(c)] if $f \in Y_+^{\pm 1}$ and $l^\prime(com(f,h f)) > 0$ for some
$h \in \langle Y_0 \rangle$ then $l^\prime(com(f,hf)) = l^\prime(f)$.
\item[(d)] There are finitely many elements $\{h_1,h_2,\dots,h_m\}$ so that $l^\prime(h_i)=0$ for all $i=1,2,\dots,m$,
$T=Y\cup\{h_1,h_2,\dots,h_m\}$ satisfies the conditions (a),(b) and (c), and the following condition holds:
if $f \in Y_+^{\pm 1}$ and $l^\prime(com(f,hf)) > 0$ for some
$h \in \langle T_0 \rangle$ then $l^\prime(com(f,hf)) = l^\prime(f)$ and
$f^{-1}hf \in \langle T_0 \rangle$.
\end{itemize}
\end{defn}

\begin{rem} According to \cite[Lemmas 10,11,12 and Proposition 2]{KharlampovichMyasnikov10}, an arbitrary generating set $S$ of a regular ${\mathbb Z}^n$-free group $G$ can be transformed into a reduced generating set $W$ for $G$. Then, applying \cite[Theorem 6]{KharlampovichMyasnikov10} to $W$, by induction on the number of the HNN-extensions in Theorem~\ref{thm:main2}, one can obtain a reduced generating set $Z=X\cup\{t_1,\dots,t_l\}$ for $G$, where where $X$ is a free basis for the free group $G_1$ and $t_1,\dots,t_l$ are stable letters.
\end{rem}

\begin{lem} \label{lem:reductioninnormalforms} Let $G$ be a $\mathbb{Z}^n$-free complete group obtained by a finite number of HNN extensions from a finitely generated free group $F(X)$, where $X$ is a free basis. Fix a reduced  generating set $X\cup\{t_1,\dots,t_l\}$ for $G$; $t_1,\dots,t_l$ are stable letters. Let $w_1,\dots,w_r$ be a finite number of words in normal forms. If for some of these words $w=u_1t_{i_1}u_2t_{i_2}\dots u_m t_{i_m}$ $(u_j\in F(X))$ its length $l(w)$ is strictly smaller than the sum of the lengths of the letters then we can replace $t_1,\dots,t_l$ by new generators $\bar{t}_1,\dots,\bar{t}_l$ so that $w=\bar{u}_1\bar{t}_{i_1}\bar{u}_2\bar{t}_{i_2}\dots \bar{u}_m \bar{t}_{i_m}$ is a normal form and
$$l(w)=\sum_{j=1}^m (l(\bar{u}_j)+l(\bar{t}_{i_j})).$$
New generators can be chosen so that the equality holds for all the words $w_1,\dots,w_r$, simultaneously. Moreover, the new generating set is also reduced.
\end{lem}
\begin{proof} W.l.o.g., we can assume that $l(w_j)$ is strictly smaller than the sum of the lengths of the letters in its normal form, for all $j$. Let $S=\{s_1,\dots,s_m\}$ be the set of all the stable letters of maximal height $N\leq n$ that occur in the words $w_j$, $j=1,\dots,r$; we order the elements of $S$ arbitrarily. Let $s=s_1$ conjugate $C=\langle v,t_1,\dots,t_k\rangle$ to $\phi(C)$. We write this as $s^{-1}Cs=\phi(C)$, or $s\phi(C)s^{-1}=C$. Then by the choice of $S$ and by Theorem~\ref{thm:main2},
$$ht(s)>ht(t_k)>\dots>ht(t_1)>ht(v).$$

We fix $j$ and in the words $w_j$ and $w_j^{-1}$ find all the occurrences of subwords of the form $d^\alpha s^{-1} f^\beta$ with $d\in C$, $f\in\phi(C)$, $\alpha\neq 0$ and $\beta\neq 0$, so that $l(d^\alpha s^{-1})<l(d^\alpha) + l(s)$ and $l(s^{-1} f^\beta) < l(s) + l(f^\beta)$.
% By Lemma~\ref{lem:axisofstableletter}, those $d$'s and $f$'s are distinct for distinct $s\in S$.
Note that for these inequalities to hold, $d^\alpha$ and $s^{-1}$, as well as $s^{-1}$ and $f^\beta$ should not be coherent. If $s^{-1}$ and $t_k^{-1}$ are coherent, this implies that $d^\alpha$ and $t_k^{-1}$, as well as $t_k^{-1}$ and $f^\beta$ should not be coherent, either. If there is no such $\alpha$ (or no such $\beta$) then we set $\alpha=0$ (or $\beta=0$).

Recall that by Lemma~\ref{lem:axisofstableletter}, $[x,d^px]\subset [x,sx]$ and $[x,f^{-p}x]\subset [x,s^{-1}x]$ either for every  positive, or for every negative integer $p$, and that $[x,t_k^qx]\subset [x,sx]$ and $[x,\phi(t_k)^{-q}x]\subset [x,s^{-1}x]$ either for every  positive, or for every negative integer $q$. Since $t_k$ is an element of the largest height in $C$, there exist $\gamma,\delta$ so that
\begin{equation}\label{e:subsegment}
[x,d^{\alpha} x]\subseteq [x,t_k^\gamma x],\quad [x,f^{\beta} x]\subseteq [x,(\phi(t_k))^\delta x].
\end{equation}
For all the occurrences of subwords $d^\alpha s^{-1} f^\beta$ in $w_j^{\pm 1}$, we choose $\gamma_j$ and $\delta_j$ minimal in the absolute value with respect to the inclusions~(\ref{e:subsegment}).

If all $\alpha=0$ (or all $\beta=0$) then we set $\gamma_j=0$ (or $\delta_j=0$). Let $$K=\max_j\{|\gamma_j|\}, \quad M=\max_j\{|\delta_j|\}.$$
If $K=0$ then we seek a subword $hs$ with $h\notin C$ so that $l(hs)<l(h)+l(s)$. If there is such a subword in some $w_j^{\pm 1}$ then we change the value of $K$ to $K=1$. Similarly, if $M=0$ we seek subwords $sh$ with $h\notin C$ so that $l(sh)<l(h)+l(s)$ and set $M=1$ if we find one. If at least one of $K,$ $M$ is not 0, in the generating set we replace $s$ by the new letter $\bar{s}$, where $s=t_k^K\bar{s}\phi(t_k)^M$. Since in this latter product both pairs of adjacent letters translate in the same direction and $l(t_k)=l(\phi(t_k))$, it follows that
$$l(s)=l(t_k^K)+ l(\bar{s}) + l(\phi(t_k)^M)\ \Rightarrow\ l(\bar{s})= l(s)-(K+M)l(t_k).$$
We express $s$ in terms of $\bar{s}$ in every word $w_j^{\pm 1}$ and freely reduce the words whenever possible. Note that since the generating set is reduced, the free reductions only affect infinitesimal elements and never affect $\bar{s}$; in particular, the words all remain in normal forms. We also note that $ht(t_k)\geq ht(d),ht(f)$, and whenever $ht(t_k)= ht(d)$ (or $ht(t_k)=ht(f)$), both $t_k^K$ and $d^\alpha$ (or both $\phi(t_k)^M$ and $f^\beta$) cancel in the free reduction.

We proceed with $s=s_2$, then $s=s_3,\dots,s_m$, in the same way as above. Now, we consider all the stable letters of height $N-1$ and apply the same procedure (to the modified words $w_j$) to replace them by new generators, whenever needed. Note that if a stable letter $s$ in question belongs to an abelian subgroup then we can still apply the same procedure, the only difference is that $\phi(C)=C$ and $\phi$ is the trivial automorphism. The letters $t_k$ in the abelian subgroups are replaced by $t_{k-1}^K\bar{t}_k t_{k-1}^M$ and $t_1$ is replaced by $v^K\bar{t}_1v^M$, whenever applicable.

Having done with all the letters of height 2, we continue with $s=\bar{s}_1$, $s=\bar{s}_2$, etc., and keep going until free cancellations are no longer possible. The process will be finite because in every new round, when we come back to $\bar{s}_i$, the cancellations only affect infinitesimal factors of the height, which is higher than before. That is, if $l(d s_i)<l(d) + l(s_i)$ first, and then we have $l(\bar{d} \bar{s}_i)<l(\bar{d}) + l(\bar{s}_i)$, it follows that $ht(d)<ht(\bar{d})$.

All the words remain in normal forms. The length of the product of two elements is smaller than the sum of the lengths of the factors only in the case when their axes overlap and the elements translate in the opposite directions on the common segments. In the process above we remove all such occasions while keeping track of the changes of the lengths of the letters involved. The equality of the lengths from the statement of the lemma follows. The generating set remains reduced because every change of a stable letter only affects an infinitesimal part of it.
\end{proof}

The following statement is used in section ~\ref{sec:gluing} below.
\begin{thm} \textbf{(A single HNN-extension)}
\label{thm:singleHNN}
Let $G$ be a finitely generated group with a regular free Lyndon
length function in $\mathbb{Z}^n$. Then $G$ can be represented as
a union of a finite series of groups
$$G_1 < G_2 < \cdots < G_r = G,$$
where $G_1$ is a free group of finite rank, and $\forall i=1,\dots,r-1$
$$G_{i+1} = \langle G_i, s_{i} \mid s_{i}^{-1}\ C_{i}\ s_{i} = \phi_{i}(C_{i}) \rangle$$
with $C_{i}$ and $\phi_{i}(C_{i})$ maximal abelian subgroups of $G_i$,
$\phi_{i}$ an isomorphism, and so that the following conditions are satisfied:
\begin{enumerate}
\item[(1)] $\ell(\phi_{i}(w)) = \ell(w)$ $\forall w \in C_{i}$,

\item[(2)] $\phi_{i}(w)$ and $w^{-1}$ are not conjugate in $G_i$, for any $w \in C_{i}$,

\item[(3)] $C_{i} = \langle c^{(i)}_1, \ldots, c^{(i)}_{m_{i}} \rangle,$ with
$$ht(c^{(i)}_k) <  ht(c^{(i)}_{k+1}),\ 1\leq k\leq m_i-1, \quad
ht(\phi_{i}(c^{(i)}_j)) =  ht(c^{(i)}_j),\ 1\leq j\leq m_i;$$

 \item[(4)] Generators of each $C_i$ can be chosen as follows: $c_1^{(i)}$ is a cyclically reduced element of $G_{j_1-1}$, and
 $c_2^{(i)}=s_{j_1},\ldots ,c_m^{(i)}=s_{j_{(m-1)}},$ where \newline $2\leq j_1,\ldots, j_{(m-1)}<i$.

\item[(5)]  Note that $C_{1}, \dots,  C_i, \phi_{1}(C_{1}), \dots, \phi_{i}(C_{i})$ are subgroups of $G_i$, $\forall i$. If $A, B \in \{C_{j}, \phi_{j}(C_{j}), \mid\ 1\leq j\leq i\}$ then either $A = B$, or $A$ and $B$ are not conjugate in $G_j$;

%\item[(6)] $C_{i}$ cannot appear on the list
%$$\{ C_{1}, \phi_{1}(C_{1}), \dots,  C_{n}, \phi_{n}(C_{n}) \}$$
%more than three times. Moreover, if it appears three times, then necessarily $C_i=\phi _i(C_i),$ and $\phi_i$ is %the identity automorphism of $C_{i}.$

\item[(6)]\label{e:sHNN6} If $X$ is a free basis of $G_1$ then $X\cup\{s_1,\dots,s_r\}$ is a reduced generating set for $G$. By possibly replacing $s_i$,
% with $t_i=a_is_i$ for some $a_i\in C_i$, $1\leq i\leq r$,
 we obtain a reduced generating set $Y=X\cup\{y_1,\dots,y_r\}$ for $G$, so that the length in ${\mathbb Z}^n$ of every generator of each centralizer $C_{i}$ is the sum of the lengths of elements $g^{(i,j)}_l\in Y^{\pm 1}$, where
\[
c_j^{(i)}=g^{(i,j)}_1\dots g^{(i,j)}_{m_{ij}}, \quad l(c_j^{(i)})=l(g^{(i,j)}_1)+\dots+ l(g^{(i,j)}_{m_{ij}}),\ 1\leq l\leq m_{ij},\ \forall i,j.
\]
\end{enumerate}
\end{thm}

\begin{proof} Statements (1),(2),(3) and (5) follow from theorem~\ref{thm:main2}.

To prove (4) and (6), we start with a reduced generating set $Z=X\cup\{t_1,\dots,t_l\}$ for $G$, where $X$ is a free basis of the free group $G_1$. We assume that the generating set of $G$ satisfies the conclusions of Lemma~\ref{lem:reductioninnormalforms}.

Proof of (4).  We distinguish the following two types of the HNN-extensions:
\begin{enumerate}
  \item Extensions of centralizers, when $\phi$ is the identity automorphism of the associated subgroup.
  \item Extension of conjugacy classes, when $C_i$ and $\phi(C_i)$ are not conjugate in the base group $G_i$.
\end{enumerate}
 Let $U=\{u_1,\dots,u_b\}\subset G_1$ be shortest representatives of those conjugacy classes of elements of the free group whose centralizers are extended in $G$. Here ``shortest" means shortest with respect to the length function; note that in the free group $G_1$ $l(u)$ equals the word length $|u|$ of $u$. So,
$$C_{G_1}(u_i)\subsetneqq C_{G}(u_i)\quad i=1,2,\dots,b,$$ and our choice of the elements of $U$ ensures that the centralizer $C_{G_1}(u_i)=\langle u_i\rangle$ is cyclically reduced. By a series of HNN-extensions we extend the centralizer $C_{G_1}(u)$ to $C_{G}(u)=\langle u,t_1,\dots,t_m\rangle$, for each $u\in U$; the sets of stable letters for distinct elements of $U$ are disjoint, and their number $m$ depends on $u$. By Theorem~\ref{thm:rankabelian}, we can  assume that the inequality $ht(t_j)<ht(t_{j+1})$ holds for all $j=1,2,\dots,m-1$. Furthermore, consider all those $u\in U$ whose conjugacy classes $[u]_{G_1}$ in $G_1$ are extended in $G$, so that $[u]_{G_1}\subsetneqq[u]_{G}$. For every such $u$, we consider in $G_i$ the elements of $G_i$ that are conjugate to $u$ in $G$ but not in $G_i$. These are the elements $v$ that satisfy both $[u]_{G_i}\neq[v]_{G_i}$ and $[u]_G=[v]_G$. Let $V_i=\{v^{(i)}_1,\dots,v^{(i)}_{q_i}\}$ be the set of shortest representatives of those conjugacy classes $[v]_{G_i}$ in $G_i$. We call $u$ and the elements of $V_i$ \emph{principal elements}. Note that $V_r=\varnothing$ and $$V_1\supseteqq V_2\supseteqq\dots\supseteqq V_{r-1}.$$
Also note that every centralizer in $G$ is an abelian subgroup of $G$, and every cyclically reduced centralizer is the  centralizer of a cyclically reduced element, according to Lemma~\ref{lem:reducedcentralizer}. It follows that in every associated subgroup of the HNN-extensions in the statement of Theorem~\ref{thm:main2} one can choose a principal element as above.

Theorem~\ref{thm:main2} describes how $G$ is obtained from $G_1$ by a sequence of HNN-extensions. We slightly deviate from the procedure in that at every step $i$ we choose the associated subgroups of HNN-extensions to be \emph{all} those, and only those, that are the centralizers of \emph{principal} elements. More precisely, whenever we extend the centralizer of $u$ at a step $j$, we extend the centralizer of $v$ for every $v\in V_{j+1}$, so that $C_{G_j}(u)\simeq C_{G_j}(v)$, which we do not necessarily do when following the process of Theorem~\ref{thm:main2}. (Note that if $v\in[u]_G$ then $C_{G}(v)\simeq C_{G}(u)=\langle u,s_1,\dots,s_m\rangle$.)
Eventually, we extend the centralizer $C_{G_1}(v)$ to $C_{G}(v)=\langle v,t_1,\dots,t_m\rangle$, for each $v\in V_1$; the sets of stable letters are all disjoint. We shall always assume that $[\tilde{u} s_j]_G=[\tilde{v} t_j]_G$, for some $\tilde{u}\in\langle u,s_1,\dots,s_{j-1}\rangle$ and $\tilde{v}\in\langle v,t_1,\dots,t_{j-1}\rangle$, for all $j$.

Having extended the conjugacy classes of elements from $G_1$, we apply the same procedure to $G_2$, then to $G_3$, etc. For each one of these $G_i$, we look for elements whose centralizers do not contain elements of height less than $i$; shortest representatives of their conjugacy classes will have height $i$.

By changing the process, we might have added more stable letters than appear in Theorem~\ref{thm:main2}. We need to determine the value of the length function  $l$ at each one of the new elements. The new generators, that we might have introduced, correspond to existing elements of $G$ written as products in the old generators, and we  assign the lengths of the new generators correspondingly. However, for the assertion (6) to be true, we need to achieve a little more. In some cases, we will replace the obtained generators of $C_{i,j}$ by  new stable letters whose length will be adjusted.

We proceed as follows. Let $C_{i,j}$ from the statement of theorem~\ref{thm:main2} be conjugate to $C_{G}(u)$ in $G$. Fix a generator $c_k^{(i,j)}$ of $C_{i,j}$. It follows from the construction that  $c_k^{(i,j)}=wp$ for some stable letter $p$ and $w\in C_{i,j}$ so that $l(w)\ll l(p)$. We collect all the stable letters $p_1,p_2,\dots$ that appear as factors in the conjugates of $c^{i,j}$ in all the HNN-extensions with $C_{i,j}$ as an associated subgroup. In the finite set of elements, that we obtain in this way, we find an element of the shortest length $L_t$. In $\langle u,s_1,\dots,s_{k}\rangle$, there is a unique element $y$ of length $l(y)=L_t$; we choose this element to be a new generator: $y_k:=y$. We do this for all the generators in every centralizer $C_G(u)$ with $u\in U$. Then we express all the stable letters in the generating set of $G_i$ in terms of the new stable letters, and adjust the elements of $G$ accordingly. Note
  that since we have chosen $y_k$ to have the minimum length, every stable letter $p$ in the conjugates of $C_{i,j}$ with $ht(p)=ht(y_k)$ will be replaced by the product $\tilde{u}y_k$ where for all $l$, $[x,\tilde{u}^lx]\subset A_{y_k}$ and $\tilde{u}$ translates points of those common segments to the same direction as $y_k$ does.

The generating set $Y$ that we obtain is different from the generating set $Z$ obtained in theorem~\ref{thm:main2}; for instance, $Y_0$ may contain more elements than $Z_0$. However, the elements of $Z_+$ and those of $Y_+$ can only differ by infinitesimal factors, hence $Y$ is also reduced.

(6) By Lemma~\ref{lem:reductioninnormalforms}, we can assume that for $w_1,\dots,w_r$, whose centralizers are extended in $G$, the length $l(w_i)$ is the sum of the lengths of the old letters. The choice of the new generators for the abelian subgroups $C_{i,j}$ of $G$ does not lead to any length reduction in the words $w_1,\dots,w_r$. Therefore $l(w_j)$ remains equal to the sum of the lengths of the new letters.
\end{proof}

\subsection{Weighted word metric}
%%%%%%%%%%%%%%%%%%%%%%%%%%%%%%%%%%%%%%%%%%
%%%%%%  Key Lemma
%%%%%%%%%%   System of linear equations
%%%%%%%%%%%%%%%%%%%%%%%%%%%%%%%%%%%%%%%%%%%%%%%%%%%%%%%%%%%%%%%%%%%%%%%%%%

Let $G$ be a group equipped with a length function $l$, and let $Y$ be a finite generating set for $G$. In general, the equality $l(u)=l(v)$ for $u,v\in G$ does not imply that the length of $u$ equals the length of $v$ in the word metric on $G$. However, for a regular $\mathbb{Z}^n$-free group $G$, given a finite set of pairs of elements $(u_i,v_i)$ of $G$ with $l(u_i)=l(v_i)$ for all $i$, we can assign positive integer weights to the generators in $Y^{\pm 1}$ so that the length of $u_i$ will equal the length of $v_i$ in the weighted word metric $wm\colon G\rightarrow\mathbb{Z}^+\cup\{0\}$. More precisely, suppose we have assigned a weight $wm(g)>0$, to every generator $g\in Y$. We set $wm(1)=0$ and $wm(g^{-1})=wm(g)$ by definition, and for an arbitrary $h\in G$ we define the weighted length of $h$ as follows:
\[
wm(h)=\min_{g_1\dots g_{m}=h}\left\{\sum_{i=1}^{m}wm(g_i)\right\},
\]
where the minimum is taken over all the words in the free group on $Y$ that represent $h$. It is straightforward to check that $wm$ is indeed a metric.

There is a weighted Cayley graph of a group $G$, associated with the weighted metric $wm$. In the weighted graph the length of an edge labeled by a generator $g\in Y$ is $wm(g)$. For an arbitrary element $h\in G$, its weighted length is the shortest (weighted) distance from 1 to $h$ in the weighted Cayley graph of $G$.

In the proof of Key Lemma below we use the following result.
\begin{prop}\cite{GK} \label{GK} Any two nontrivial ordered abelian groups
satisfy the same existential sentences in the language
$L=\{0,+,-,<\}.$
\end{prop}

\begin{lem} \textbf{(Key Lemma)} Let $G$ be a a regular $\mathbb{Z}^n$-free group, and let $Y=X\cup\{t_1,\dots,t_k\}$ be a generating set for $G$ as in theorem~\ref{thm:singleHNN}(6). One can assign positive integer weights to the generators of $G$ so that in every HNN-extension $t^{-1}Ct=\phi(C)$ in the statement of theorem~\ref{thm:singleHNN}, for every element $h\in C$, $wm(h)=wm(\phi(h))$.
\end{lem}
\begin{proof} Note that the associated subgroups are free abelian groups and the generating set of each one of them, introduced in theorem~\ref{thm:singleHNN}(4), is a free basis. Therefore, it suffices to arrange for the weight of every generator $u\in C$ be equal the weight of its image $v=\phi(u)$. Let $U$ denote the set of all the generators of the associated subgroups of all the HNN extensions; $U$ is a finite set. If $v,u\in U$ are as above then $l(u)=l(v)$, by the assertion~\ref{thm:singleHNN}(1). Let the words $u=g_1\dots g_{m_u}$ and $v=f_1\dots f_{m_v}$, with $g_i,f_j\in Y^{\pm 1}$, be as in theorem~\ref{thm:singleHNN}(6), so that
\[
l(u)= \sum_{i=1}^{m_u} l(g_i),\quad l(v) = \sum_{i=1}^{m_v} l(f_i).
\]
 Therefore, we can write the following finite system of equations and inequalities that hold in $\mathbb{Z}^n$:
\begin{eqnarray*}
 \sum_{i=1}^{m_u} l(g_i) = \sum_{i=1}^{m_v} l(f_i);  \qquad  l(g)>0\ \forall g\in Y,
\end{eqnarray*}
where $u$ and $v$ run over all the (pairs of) elements of $U$. Therefore, we have a finite set of formulas in the language $\{0,+,-,<\}$ satisfiable in the nontrivial ordered abelian group $\mathbb{Z}^n$. Note that we use no constants when writing these formulas.  Proposition~\ref{GK} implies that the formulas are satisfiable in $\mathbb{Z}$. This allows us to define a weighted word metric $wm\colon G\rightarrow\mathbb{Z}^+\cup\{0\}$ so as to have the system of equations and inequalities in $\mathbb{Z}$, as follows:
\begin{eqnarray}\label{e:equation}
\sum_{i=1}^{m_u} wm(g_i) = \sum_{i=1}^{m_v} wm(f_i);\qquad wm(g)>0\ \forall g\in Y.
\end{eqnarray}
 The values of $wm$ on the generating set $Y$ are chosen to satisfy the equations~(\ref{e:equation}).

It remains to show that
\begin{eqnarray}\label{e:wm}
wm(u)=\sum_{i=1}^{m_u} wm(g_i), \quad wm(v) = \sum_{i=1}^{m_v} wm(f_i), \quad \forall u,v\in U.
\end{eqnarray}
We argue by induction on the number $r$ of HNN-extensions, with the notation of theorem~\ref{thm:singleHNN}. Clearly, the equalities~(\ref{e:wm}) hold for $u,v\in U\cap G_1$, where $G_1$ is the group freely generated by $X\subset Y$ (we use the notation of theorem~\ref{thm:singleHNN}).
Now, assume that the equalities~(\ref{e:wm}) hold for all $u',v'\in U\cap G_r$. Recall that $G_{r+1}=\langle G_r,t\mid t^{-1}Ct=\phi(C)\rangle$. Let $u\in U\cap G_{r+1}$, and let $u=g_1\dots g_{m_u}$ be the normal form for $u$ in $G_{r+1}$, with shortest coset representatives of the associated subgroups. Suppose $u=g_1\dots g_{m_u}=h_1\dots h_{m}$, with $g_i, h_j\in Y$ for all $i,j$, where $h_1\dots h_{m}$ may not be a normal form for $u$. We claim that
\[
\sum_{i=1}^{m_u} wm(g_i)\leq \sum_{j=1}^{m} wm(h_j).
\]
Our proof shows that as we obtain the normal form $g_1\dots g_{m_u}$ for $u$ from $h_1\dots h_{m}$, at each step the weighted length of the product either becomes shorter or stays the same. Note that the equality $wm(w)=wm(\phi(w))$ holds for the generators of $C$, due to our choice of the function $wm$. Furthermore, since $C$ is a free abelian group and the generators of $C$ form a basis, the equality holds for all $w\in C$. Every one step reduction of $h_1\dots h_{m}$ is one of the following:
\begin{itemize}
  \item Free cancellation, which apparently decreases the weighted length.
  \item If subwords $t^{-1}wt$ or $t\phi(w)t^{-1}$ with $w\in C$ occur in the product, replace them with $\phi(w)$ or $w$, correspondingly. The weighted length decreases.
  \item Replace $t^{-1}w$ with $\phi(w)t^{-1}$ and $t\phi(w)$ with $wt$. The weighted length does not change.
  \item If $d_1$ and $d_2$ are two representatives of right cosets of $C$ (or $\phi(C)$), replace $d_1d_2$ with $d$ where $dC=d_1d_2C$ (or $d\phi(C)=d_1d_2\phi(C)$). Note that $d_1,d_2$ and $d$ are all in $G_r$, so by the induction hypothesis, $d$ is a shortest representative of its coset. Therefore, the reduction does not increase the weighted length.
\end{itemize}
\end{proof}

\section{Gluing CAT(0) spaces} \label{sec:gluing}

In this section, following the strategy introduced in~\cite{AlibegovicBestvina}, we prove that $\mathbb{Z}^n$-free groups are CAT(0).

 The first ingredient of the proof is the following statement.

\begin{prop} \cite[Proposition II.11.13 and Corollary II.11.14]{BridsonHaefliger}  \label{prop:gluing} Let $A$ be a compact locally CAT(0) metric space.
\begin{enumerate}
  \item Let $X$ be a locally CAT(0) metric space. If $\varphi,\phi\colon A\rightarrow X$ are local isometries, then the quotient of $X\coprod(A\times[0,1])$ by the equivalence relation generated by $[(a,0)\sim\varphi(a); (a,1)\sim\phi(a)],$ for all $a\in A$, is locally CAT(0).
  \item Let $X_0$ and $X_1$ be locally CAT(0) metric spaces. If $\varphi_i\colon A\rightarrow X_i$ is a local isometry for $i=0,1$, then the quotient of $X_0\coprod(A\times[0,1])\coprod X_1$ by the equivalence relation generated by $[(a,0)\sim\varphi_0(a);
(a,1)\sim\varphi_1(a)],$ for all $a\in A$, is locally CAT(0).
\end{enumerate}
\end{prop}

To state the following theorem, we should remind the definition of a geometrically coherent space, inspired by the work of Wise~\cite{WiseSectionalCurv}, and introduced in \cite{AlibegovicBestvina}.

\begin{defn}\label{defn:coherentspace} \textbf{(Geometrically coherent space)} Let $X$ be a connected locally CAT(0) space, and let $C$ be a connected subspace of $X$. $C$ is a \emph{core} of $X$ if $C$ is compact, locally CAT(0), and the inclusion $C\hookrightarrow X$ induces a $\pi_1$-isomorphism.

Let $Y$ be a connected locally CAT(0) space. $Y$ is called \emph{geometrically coherent} if every covering space $X\rightarrow Y$ with $X$ connected and $\pi_1(X)$ finitely generated has the following property. For every compact subset $K\subset X$ there is a core $C$ of $X$ containing $K$.
\end{defn}
%%%%%%%%%%%%%%%

\begin{defn} \textbf{(Maximal separated torus)}
 Let $U$ be a connected geodesic metric space and $A\subset U$ be a $k$-torus, so that $A=c_1\times c_2\times\dots\times c_k$ for some non-homotopic simple closed curves $c_1, c_2,\dots, c_k\subset U$. We say that $A$ is a \emph{maximal torus in} $U$ if $A$ is not contained in any bigger torus $T\subset U$. A maximal torus $A$ is \emph{separated in} $U$ if for any torus $\hat A\subset U$ so that $A\cap\hat A\neq\emptyset$ it follows that $\hat A\subseteq A$.
\end{defn}
We are interested in the case when every torus of $U$ is contained in a maximal torus, and all tori maximal in $U$ are separated in $U$. In the language of fundamental groups, this means that every abelian subgroup of $\pi_1(U)$ is contained in a maximal abelian subgroup, and maximal abelian subgroups are conjugacy separated; in this case $\pi_1(U)$ is called a CSA group. It follows from~ \cite[Proposition 4.13]{Bass} (cf. also~\cite[Chapter 3, Lemma 3.9]{ChiswellBook}) that $\mathbb{Z}^n$-free groups are CSA.

In the statement below, by a 1-torus we mean a simple closed curve.

\begin{thm}\label{thm:finitecore} Let $U$ be a geometrically coherent locally CAT(0) space, let $T$ be a $k$-torus for some integer $k\geq 1$, and let $\varphi\colon T\rightarrow U$ and $\phi\colon T\rightarrow U$ be local isometries. Furthermore, assume that both images $\varphi(T)$ and $\phi(T)$ are locally convex, maximal and separated in $U$. We also assume that if $\varphi(T)\neq\phi(T)$ then the $\pi_1(U)$-orbits of $\varphi(T)$ and $\phi(T)$ are distinct.
%
%that the $\pi_1(U)$-orbits of $\varphi(T)$ and of $\phi(T)$ are disjoint.
%
 Let $Y$ be the quotient of
\[
 U\coprod T\times[0,1]
\]
by the equivalence relation generated by $$[(a,0)\sim\varphi(a),\ (a,1)\sim\phi(a),\forall a\in T].$$
Then $Y$ is geometrically coherent.
\end{thm}

\begin{proof} Let us be given a covering space $p\colon X\rightarrow Y$ with $X$ connected and $\pi_1(X)$ finitely generated, and a compact subset $K\subset X$. Then $X$ can be viewed as a graph of spaces with vertex spaces the components of $p^{-1}(U)$ and edge spaces the components of $p^{-1}(T\times{\frac 12})$.
Let $G=\pi_1(X)$ be the fundamental group of $X$, and let $\Gamma$ be the associated graph of groups. Then $\pi_1(\Gamma)=\pi_1(X)=G$; note that $G$ is finitely generated. Every edge group in $\Gamma$ is either free abelian of rank less or equal to $k$, or trivial. Each edge space is $E=A\times [0,1]$ where $A$ is the product of lines and circles:
$$A=c_1\times c_2\times\dots\times c_p\times\mathbb{R}^q,$$
where $p+q=k\geq 0$ and each $c_i$ is a circle. In the vertex space, circles are identified with simple closed curves, and lines are identified with lines, so that for every edge space $E$, each one of $A\times\{0\}$ and $A\times\{1\}$ is identified with either a torus, the product of a torus with lines, or the product of several lines. It may happen that both $A\times\{0\}$ and $A\times\{1\}$ are glued to the same vertex space.

Note that our assumptions on the images $\varphi(T)$ and $\phi(T)$ of $T$ imply that if $\varphi(T)\neq\phi(T)$ then the boundary components of distinct edge spaces in $X$ are identified with disjoint subsets of the vertex spaces. This is because maximal tori of $U$ are separated in $U$ and the $\pi_1$-orbits of $\varphi(T)$ and $\phi(T)$ are disjoint.

First, assume that the graph of spaces $X$ is finite; in particular, $X$ contains finitely many vertex and edge spaces. In this case, the graph of groups $\Gamma$ is also finite. Note that every edge group of $\Gamma$ is finitely generated. It follows that every vertex group of $\Gamma$ is finitely generated as well. Therefore, in any vertex space of $X$, every compact subset is contained in a core. Given the compact set $K$ in $X$, to find a core $C\supset K$ in $X$, we first select the product of circles $c_j$ and segments $[a_l,b_l]$ with the unit interval $I=[0,1]$
$$B_E=c_1\times\dots\times c_p\times[a_1,b_1]\times\dots\times[a_q,b_q]\times[0,1]$$
inside each product $E=c_1\times\dots\times c_p\times\mathbb{R}^q\times[0,1]$ of $p$ circles and $q$ lines with $I$, so that $K\cap E\subset B_E$.
For every vertex space $V$ of $X$ we choose a core $C(V)$ so that it contains the intersection $V\cap K$ with $K$; $C(V)$ contains the intersection $V\cap B_E$ for every edge $E=A\times I$ in $X$; and finally, $C(V)$ contains the intersection $V\cap T$ for every generalized annulus $GA=T\times I$ in $X$.
The core $C$ is the union of the following spaces:
\begin{itemize}
 \item the cores $C(V)$ of all vertex spaces,
 \item all the generalized annuli $GA$ in $X$, and
 \item all the products $B_E$ in $X$.
\end{itemize}
 $C$ is clearly compact as the finite union of compact spaces. Note that the disjoint union of the cores $C(V)$ of all vertex spaces is a locally CAT(0) metric space, and that every generalized annulus as well as every product $B_E$ is a compact locally CAT(0) space. It follows that $C$ is locally CAT(0), according to Proposition~\ref{prop:gluing}. $C$ is connected by construction. Furthermore, the graph of groups $\Gamma_C$ corresponding to $C$ has the vertex groups isomorphic to those of $\Gamma$, and the same edge groups as $\Gamma$, with the same monomorphisms, so $$\pi_1(C)=\pi_1(\Gamma_C)=\pi_1(\Gamma)=\pi_1(X).$$

Now, consider the case when $X$ is an infinite graph of spaces. In this case the corresponding graph $\Gamma$ is also infinite. Both $X$ and $\Gamma$ are countable graphs, since $G$ is a countable group. We need a preliminary step. If $K$ intersects vertex spaces, pick an arbitrary vertex space $V$ with $V\cap K\neq \emptyset$; otherwise, choose any vertex space $V$ in $X$. Starting with $V$, build an ascending chain of finite connected subspaces $X_0\subset X_1\subset\dots\subset X$ of $X$ which exhaust $X$, and consider the corresponding subgraphs $\Gamma_0\subset\Gamma_1\subset\dots\subset \Gamma$ of $\Gamma$. The fundamental groups $G_0\subset G_1\subset\dots$ of these subgraphs form an ascending chain of subgroups of $G=\pi_1(\Gamma)$. The subgraphs $\Gamma_i$ exhaust $\Gamma$; therefore, the subgroups $G_i$ exhaust $G$. Since $G$ is finitely generated, the chain of subgroups must stabilize after finitely many steps: $$G_0\subset G_1\subset\dots G_n=G_{n+1}=\dots=G.$$ Therefore, the finite subgraph $\Gamma_n$ has the same fundamental group as $\Gamma$, and so for the corresponding subspace $X_n$ we have that $\pi_1(X_n)=\pi_1(X)$. Given a compact subset $K$, we may need to replace $X_n$ by a larger subspace $X_m$ so as to include $K$. Clearly, a core in $X_m$ is a core in $X$, so we can apply to $X_m$ the procedure described above to find a core.

\end{proof}

%%%%%%%%%%%%%%%%%%%%%%%%%%%%%%%%%%%%%%%%%%%%%%%%%%%%%%%%%%%%%%%%%%%%%%%%%%%%%%%%%%
\begin{thm}\label{thm:regularZnfreecoherent} Let $G$ be a finitely generated group with a free regular $\mathbb{Z}^n$-length function. Then $G$ is the fundamental group of a geometrically coherent space. In particular, every finitely generated subgroup of $G$ is CAT(0).
\end{thm}
\begin{proof} We build a sequence of spaces corresponding to the chain of HNN-extensions in theorem~\ref{thm:singleHNN}. Each space is a 2-dimensional presentation complex of the corresponding group.
$U_1$ is the wedge of circles corresponding to the free group $F=G_1$. Note that every maximal abelian subgroup $C$ of $G_1$ is cyclic; let $C=\langle c^{(1)}\rangle$. Let $\alpha$ and $\beta$ be the reduced paths in $U_1$ corresponding to $c^{(1)}$ and $\phi(c^{(1)})$. Note that the lengths $l(\alpha)$ and $l(\beta)$ are equal. We rescale the metric on $S^1$ so as to make the total length of the circle equal $l(\alpha)$,  and let $\tau\colon S^1\rightarrow \alpha$ and $\sigma\colon S^1\rightarrow \beta$ be local isometries. $U_2$ is the quotient of
\[
 U_1\coprod S^1\times[0,1]
\]
by the equivalence relation generated by $[(a,0)\sim\tau(a),\ (a,1)\sim\sigma(a),\forall a\in S^1]$.
If $c^{(1)}\neq \phi(c^{(1)})$, $c^{(1)}$ and $\phi(c^{(1)})$ are not conjugate in $G_1$; therefore, in this latter case  the $\pi_1(U_1)$-orbits of $\alpha$ and $\beta$ are disjoint. It can be readily seen that all the conditions of Theorem~\ref{thm:finitecore} are satisfied. It follows that $U_2$ is geometrically coherent.

To proceed, assume that $G_i=\pi_1(U_i)$, where $U_i$ is a geometrically coherent space. Let $c^{(i)}_1, \ldots, c^{(i)}_{m_{i}}$ be generators of a maximal abelian subgroup $C_i$ of $G_i$, chosen so that the conditions of theorem~\ref{thm:singleHNN} are satisfied. In particular, the generators and their images under $\phi$ form bases for $C_i$ and $\phi(C_i)$.

Let $\alpha_j$ and $\beta_j$ be the reduced paths in $U_i$ corresponding to $c_j^{(i)}$ and $\phi(c_j^{(i)})$. Note that the lengths $l(\alpha_j)$ and $l(\beta_j)$ are equal. We define $$T=\prod_{j=1}^{m_i} S^1_j$$ and rescale the metric on each $S^1_j$ so as to make the total length of the circle equal $l(\alpha_j)$. This defines the metric on $T$. Let $\tau\colon S^1_j\rightarrow \alpha_j$ and $\sigma\colon S^1_j\rightarrow \beta_j$ be local isometries. The mappings $\tau$ and $\sigma$, defined on the 1-skeleton of $T$, extend to immersions and, therefore, to local isometries $T\rightarrow \tau(T)$ and $T\rightarrow \sigma(T)$. $U_{i+1}$ is the quotient of
\[
 U_i\coprod T\times[0,1]
\]
by the equivalence relation generated by $$[(a,0)\sim\tau(a),\ (a,1)\sim\sigma(a),\forall a\in T].$$
All the conditions of Theorem~\ref{thm:finitecore} are satisfied. It follows that $U_{i+1}$ is geometrically coherent.

 By induction on the number $n$ of HNN-extensions in theorem~\ref{thm:singleHNN}, we obtain the claim.
\end{proof}

By~\cite{KharlampovichMyasnikov11}, every group acting freely on a $\mathbb{Z}^n$-tree embeds by a length preserving monomorphism into a regular $\mathbb{Z}^n$-free group. Therefore, we have the following result.

\begin{thm}\label{thm:ZnfreeCAT0} Let $G$ be a finitely generated group acting freely on a $\mathbb{Z}^n$-tree. Then  $G$ acts properly and cocompactly on a CAT(0) space with isolated flats.
\end{thm}

\subsection*{Acknowledgments} The authors wish to thank Denis Serbin for helpful conversations. We are deeply grateful to the referee for the numerous remarks that helped improve the exposition considerably.
%%%%%%%%%%%%%%%%%%%%%%%%%%%%%%%%%%%%%%%%%%%%%%%%%%%%%%%%%%%%%%%%%%%%%%%%
%%                  BIBLIOGRAPHY
%%%%%%%%%%%%%%%%%%%%%%%%%%%%%%%%%%%%%%%%%%%%%%%%%%%%%%%%%%%%%%%%%%%%%%%%
\bibliographystyle{plain}
%\addcontentsline{toc}{chapter}{Reference}
%\begin{thebibliography}{99}
%\bibliographystyle{plain}

%\bibliography
%\bibliography{ZnfreeCAT0}
%\bibliography{C:\Documents and Settings\inna\My Documents\preprints\ZnfreeCAT0\Znfree}

%%%%%%%%%%%%%%%%%%%%%%%%%%%%%%%%%%%%%%%%%%%%%%%%%%%%%%%%%%%%%%%%%%%%%%%%%%%%%%%%%%%%%%%%%%%
%
\end{document}